\documentclass[a4paper,12pt]{amsart}

 \usepackage{amssymb,amsthm}
\usepackage{amscd} 
\usepackage{graphicx,color} 

\usepackage{url} 

\newtheorem{theorem}{Theorem}[section]

\newtheorem{proposition}[theorem]{Proposition}

\theoremstyle{definition}

\def\r{\mathbb R}
\def\h{\mathbb H}
\def\v{\vec{v}}

\begin{document}

\title[A Note on Helicoidal Singular Minimal Surfaces]{A Note on Helicoidal Singular Minimal Surfaces}
\author{Rafael L\'opez}
\address{ Departamento de Geometr\'{\i}a y Topolog\'{\i}a\\ Universidad de Granada. 18071 Granada, Spain}
\email{rcamino@ugr.es}


\begin{abstract} Let $\alpha\in\r$ and let $\vec{v}\in\r^3$ be a unit vector. A singular minimal surface $\Sigma$ in Euclidean space is a surface $\Sigma$ whose mean curvature $H$ satisfies $H=\alpha\frac{\langle N,\vec{v}\rangle}{\langle p,\vec{v}\rangle}$, where $N$ is the unit normal vector of $\Sigma$.  In this short note we study singular minimal surfaces which are invariant by a one-parameter  group of helicoidal motions. We prove that if $\Sigma$ is a helicoidal singular minimal surface, then the axis of the helicoidal motion is orthogonal to $\vec{v}$, $\alpha=-1$ and $\Sigma$ is a circular right cylinder.\end{abstract}

\maketitle

\section{Introduction  }

Let $\r^3$ be the Euclidean $3$-dimensional space and let   $\langle,\rangle$ be the Euclidean metric. Fix $\vec{v}\in\r^3$ be a unit vector and let $\alpha\in\r$. An orientable surface $\Sigma$ immersed in the halfspace $\r^3_+=\{p\in\r^3\colon \langle p,\vec{v}\rangle>0\}$ is called a $\alpha$-singular minimal surface if its  the mean curvature $H$  satisfies
\begin{equation}\label{eq1}
H(p)=\alpha\frac{\langle N(p),\v\rangle}{\langle p,\vec{v}\rangle},\quad p\in\Sigma,
\end{equation}
where   $N$ is the unit normal vector of $\Sigma$.  The convention for $H$ is to be the sum of the principal curvatures. When $\alpha=1$, singular minimal surfaces generalize in dimension   two, the notion of catenary because a singular minimal surface has the property   of having the lowest center of gravity among all surfaces with the same boundary curve and surface area \cite{bd,bht,dh}.  An interesting case is $\alpha=-2$. If we take $\vec{v}=(0,0,1)$, then $\Sigma$ is a singular minimal surface for $\alpha=-2$ if and only if $\Sigma$ is a minimal surface in hyperbolic space   $\h^3$ when this space is viewed in the    upper half-space model $\r_+^3$. The case $\alpha=0$ corresponds to classical minimal surfaces. From now on, we will discard the value $\alpha=0$.

In the search of examples of singular minimal surfaces, it is natural to assume some type of symmetry of the surface. If the surface is invariant by translations (cylindrical surface), then the vector $\vec{v}$ is parallel to the direction of the rulings of the surface. The classification was done by the author in   \cite{lo}. If the surface $\Sigma$ is a surface of revolution, then the vector $\vec{v}$ is parallel to the rotation axis or  the rotation axis    lies contained in the vector plane $\{p\in\r^3\colon\langle p,\vec{v}\rangle=0\}$ (\cite{lo}). The first type of  rotational singular minimal surfaces was studied in \cite{d1,d2,d3,dg,lo}.      

In Euclidean space $\r^3$, there are three types of one-parameter groups of rigid motions. Besides translations and rotations, the third one-parameter group of rigid motions is formed by helicoidal motions. A helicoidal motion is  a rotation about some line, called the   twist axis, followed by a translation parallel to that same line. Without loss of generality, we will assume that the twist axis is the $z$-axis.  
 For $h \in\r$, let $\{\mathcal{H}_t\colon t\in\r\}$ be   the one-parameter group of helicoidal motions  of $\mathbf R^3$ where $\mathcal{H}_t\colon\r^3\to\r^3$ is given by 
\begin{equation*}
\begin{split}
\mathcal{H}_t(x,y,z) &= \begin{pmatrix}\cos t&-\sin t &0\\ \sin t&\cos t&0\\ 0&0&1)\end{pmatrix}\begin{pmatrix} x\\ y\\ z\end{pmatrix}+\begin{pmatrix}0\\ 0\\ h t\end{pmatrix}\\
&=(x\cos t - y\sin t, x \sin t + y \cos t, z+ht), \quad t \in \r.
\end{split}
\end{equation*}

  A  helicoidal surface $\Sigma$ with axis the $z$-axis and pitch $h$  is a surface   invariant under $\mathcal{H}_t$ for all $t$, that is, $\mathcal{H}_t(\Sigma)=\Sigma$. If $h=0$, the surface is a  surface of revolution.  The limit case  $h\rightarrow \infty$ corresponds with   surfaces invariant under translations along the $z$-axis.  
 
 In this paper, we are interesting to classify the helicoidal singular minimal surfaces. This type of classification has been done for a variety of classes of surfaces under some  hypothesis on the curvature of the surface. Without to include all types of curvatures and ambient spaces, we mention some simple examples:  surfaces with constant mean curvature \cite{cd,pe,ri}, surfaces with constant ratio of principal curvatures (\cite{li}),  surfaces with constant anisotropic mean curvature (\cite{kp}) and  self-similar solutions of the mean curvature flow (\cite{ha}).

 In this paper, we   determine all helicoidal singular minimal surfaces. As in the case of cylindrical surfaces and rotational surfaces, this amounts to solving a second order ordinary differential equation. Although it may seem like a simple exercise, we point out that with other types of assumptions on the curvature of the surface, the corresponding ODE may turn out be difficult to manage. Taking   into account what happens for rotational singular minimal surfaces, in principle  there is no a relation between the vector $\vec{v}$ of the definition of singular minimal surface  in Eq. \eqref{eq1} and the  twist axis of the helicoidal surface. In fact, the  natural choice  that the twist axis is parallel to the vector $\vec{v}$ cannot occur unless that the surface is a surface of revolution ($h=0$). This will be proved in  Prop. \ref{pr1}. This is  because  in Eq. \eqref{eq1}, the left hand-side does not depend on  the parameter $t$ of the helicoidal motion $\mathcal{H}_t$ but  the right hand-side does. Notice that the parameter $t$ appears, in general, in the denominator $\langle p,\vec{v}\rangle$ in \eqref{eq1}. A similar situation occurs in the classification of the ruled singular minimal surfaces, where the only possible case is that the surface is cylindrical (\cite{ee}).

Definitively, we have to assume fully generality between the twist axis and the vector $\vec{v}$. In this note, the classification of the helicoidal helicoidal singular minimal surfaces is obtained in the following result.

 \begin{theorem} \label{t1}
Let $\Sigma$ be a   surface invariant by the helicoidal group $\{\mathcal{H}_t\colon t\in\r\}$ with pitch $h\not=0$.   If $\Sigma$ is a singular minimal surface with respect to $\vec{v}$, then  the twist axis is contained in the vector plane $\{p\in\r^3\colon\langle p,\vec{v}\rangle=0\}$ (in particular, $\vec{v}$ is orthogonal to the twist axis),   $\alpha=-1$  and $\Sigma$ is a circular cylinder about the $z$-axis.
\end{theorem}

  \section{Proof of Theorem \ref{t1}}

  A helicoidal surface $\Sigma$ invariant by the helicoidal group $\{\mathcal{H}_t\colon t\in\r\}$     can be obtained by moving a regular planar curve parametrized by arc-length
$$\gamma\colon I\subset\r\to\r^3,\quad \gamma(s)=(x(s),0,z(s)),$$ 
under the helicoidal motions $\mathcal{H}_t$. A parametrization of $\Sigma$ is $\Psi\colon I\times\r\to\r^3$ where
\begin{equation}\label{para}
\Psi(s,t)=\mathcal{H}_t(\gamma(s))=(x(s)\cos t,x(s)\sin t,z(s)+ht).
\end{equation}
We compute all terms of Eq. \eqref{eq1}. From now on, we will drop the variable of differentiation if it is understood in the context. The unit normal vector $N$ of $\Sigma$ is 
\begin{equation}\label{normal}
N=\frac{1}{\sqrt{x^2+h^2x'^2}}\left(hx'\sin t-xz'\cos t,-hx'\cos t-xz'\sin t,xx'\right).
\end{equation}
We now compute the mean curvature $H$ of $\Sigma$.  The coefficients of the first fundamental form of $\Sigma$ are
\begin{align*}
    E &= \langle \Psi_s, \Psi_s \rangle = 1, \\
    F &= \langle \Psi_s, \Psi_\theta \rangle = hz', \\
    G &= \langle \Psi_\theta, \Psi_\theta \rangle =  x^2 + h^2
    \end{align*}
and its coefficients of the second fundamental form are
\begin{align*}
    e &= \langle \Psi_{ss} , N \rangle = \frac{x\left( z''x' - z'x''\right)}{\sqrt{h^2x'^2 + x^2}}, \\
    f &= \langle \Psi_{s\theta} , N \rangle = -\frac{hx'^2}{\sqrt{h^2x'^2 + x^2}},\\
    g &= \langle \Psi_{\theta\theta} , N \rangle = \frac{x^2z'}{\sqrt{h^2x'^2 + x^2}}.
\end{align*}
Regularity of the surface is equivalent to $EG-F^2>0$, that is, 
$$x^2+h^2-h^2z'^2=x^2+h^2 x'^2>0$$
because $x'^2+z'^2=1$. Therefore, the mean curvature of a helicoidal surface $\Sigma$ is given by
\begin{equation}\label{mean}
 H = \frac{eG - 2fF + gE}{EG-F^2} = \frac{(x'z''-z'x'')(x^3+h^2x) + z'(2h^2x'^2+x^2)}{(h^2x'^2+x^2)^{\frac{3}{2}}}.
\end{equation}
In order to simplify all expressions, we use that $\gamma$ is parametrized by arc-length. Since $x'^2+z'^2=1$, there is a smooth function $\theta=\theta(s)$ such that
\begin{equation}\label{arc}
\begin{split}
x'(s)&=\cos\theta(s),\\
z'(s)&=\sin\theta(s).
\end{split}
\end{equation}
Moreover, the curvature of the curve $\gamma$ coincides with $\theta'(s)$. We now have from \eqref{mean} and \eqref{arc} that the mean curvature is
\begin{equation}\label{mean2}
H=  \frac{x (x^2+h^2) \theta '+\sin\theta  \left(x^2+2 h^2 \cos ^2\theta \right)}{ (x^2+h^2 \cos ^2\theta )^{3/2}}.
\end{equation}
Similarly, the unit normal vector $N$ on \eqref{normal} is now
\begin{equation}\label{normal2}
N=\frac{1}{\sqrt{x^2+h^2\cos ^2\theta}}\left(h\cos\theta\sin t-x\sin\theta\cos t,-h\cos\theta\cos t-x\sin\theta\sin t,x\cos\theta\right).
\end{equation}

Notice that 
\begin{equation}\label{reg}
EG-F^2=x^2+h^2 \cos ^2\theta>0.
\end{equation}

We now prove that the natural choice when the twist axis coincides with the vector $\vec{v}$ is not possible unless that the helicoidal motions are pure rotation motions ($h=0$).

\begin{proposition} \label{pr1}
Let $\Sigma$ be a   surface invariant by the helicoidal group $\{\mathcal{H}_t\colon t\in\r\}$ with pitch $h$.  If $\Sigma$ is a singular minimal surface with $\vec{v}=(0,0,1)$, then $h=0$.
\end{proposition}

\begin{proof} From \eqref{para} and \eqref{normal2} we have 
\begin{equation*}
\begin{split}
\langle N,\vec{v}\rangle&=\frac{x\cos\theta}{\sqrt{x^2+h^2\cos^2\theta}},\\
\langle p,\vec{v}\rangle&=\langle \Psi(s,t),\vec{v}\rangle=htz.
\end{split}
\end{equation*}
By using the expression of the mean curvature $H$ in \eqref{mean2}, equation \eqref{eq1} writes as a polynomial on $t$ of degree $1$, namely, 
$$A_0(s)+A_1(s) t=0.$$
Then the coefficients $A_0$ and $A_1$ must be zero. A computation of these coefficients gives
\begin{equation*}
\begin{split}
A_0&=z \left(x (x^2+h^2) \theta '+\sin\theta (x^2+2 h^2 \cos ^2\theta\right))- \alpha  x \cos \theta \left(x^2+h^2 \cos ^2\theta\right),\\
A_1&=h \left(x (x^2+h^2) \theta '+\sin\theta (x^2+2 h^2 \cos ^2\theta )\right).
\end{split}
\end{equation*}
The linear combination $h A_0-z A_1=0$ becomes
$$\alpha h x\cos\theta(x^2+h^2\cos^2\theta)=0.$$
Since $x^2+h^2\cos^2\theta\not=0$ by \eqref{reg}, we deduce $h=0$, or $x=0$ identically or $\cos\theta=0$ identically. If $h=0$, then the result is proved. Assume now that $h\not=0$ and we see that the cases $x=0$ or $\cos\theta=0$ identically cannot occur. 

If $x=0$ identically, then $x'=\cos\theta=0$ hence $x^2+h^2\cos^2\theta=0$, which it is not possible. If $x'=\cos\theta=0$ identically, then $x=x(s)$ is a constant function, $x(s)=x_0$ with the extra condition that $x_0\not=0$ by  regularity \eqref{reg}. In particular, from \eqref{arc} we deduce $z(s)=\pm s+z_0$. Computing again  the coefficients  $A_0$ and $A_1$, we simply have 
$A_0(s)=x_0(\pm s+z_0)$ for some $z_0\in\r$ and $A_1(s)=hx_0$. Then $A_0=0$ gives a contradiction. This completes the proof of Prop. \ref{pr1}.
\end{proof}

We prove Thm. \ref{t1} assuming that $\vec{v}$ is an arbitrary unit vector of $\r^3$. Let write the vector $\vec{v}$ in Cartesian coordinates, 
$$\vec{v}=(v_1,v_2,v_3).$$
  From \eqref{para} and \eqref{normal2}, we have
\begin{equation*}
\begin{split}
\langle N,\vec{v}\rangle&=\frac{\cos\theta (v_3 x+h v_1 \sin t-h v_2 \cos t)-x \sin\theta (v_1 \cos t+v_2 \sin t)}{\sqrt{x^2+h^2 \cos ^2 \theta  )}},\\
\langle p,\vec{v}\rangle&=\langle \Psi(s,t),\vec{v}\rangle=(h t+z)v_3+x(v_1\cos t+v_2\sin t).
\end{split}
\end{equation*}
In contrast to the proof of Prop. \ref{pr1}, equation  \eqref{eq1} now is an equation of type
$$A_0(s)+A_1(s)t+A_2(s)\sin t+A_3(s)\cos t=0.$$
Since the functions $\{1,t,\sin t,\cos t\}$ are linearly independent, the coefficients $A_i(s)$ must be zero. The computation of these coefficients yields
\begin{equation*}
\begin{split}
A_0&=-v_3 \left(\alpha x\cos\theta(x^2+ h^2  \cos ^2\theta)-xz\theta'(h^2+x^2)-x^2 z \sin \theta-2 h^2 z \sin \theta \cos ^2\theta\right), \\
A_1&=h v_3 \left(  x \theta '(h^2+x^2)+x^2 \sin \theta+2 h^2 \sin \theta \cos ^2\theta\right), \\
A_2&= \left( \alpha x\sin\theta(x^2+ h^2  \cos ^2\theta)+x^2\theta'(x^2+h^2) +2 h^2   x \sin \theta \cos ^2\theta + x^3 \sin \theta\right)v_2\\
& -\alpha h\cos\theta(   x^2  +   h^2   \cos ^2\theta)v_1,\\
A_3&= \left( \alpha x\sin\theta(x^2+ h^2  \cos ^2\theta)+x^2\theta'(x^2+h^2)+2 h^2  x \sin \theta \cos ^2\theta  +  x^3 \sin \theta\right)v_1\\
&+ \alpha  h \cos\theta(x^2 +   h^2  \cos ^2\theta)v_2.\\
\end{split}
\end{equation*}
The linear combination $hA_0-z A_1=0$ writes simply as
$$ \alpha  h v_3x \cos\theta   (x^2+h^2 \cos ^2 \theta  )=0.$$ 
By regularity \eqref{reg}  and because $h\not=0$, we deduce 
$$v_3x \cos\theta=0.$$
We distinguish three cases.
\begin{enumerate}
\item Case  $x=0$ identically. Then $\cos\theta=0$ and regularity \eqref{reg} is lost. This case cannot occur. 
\item Case $\cos\theta=0$ identically. Then $x(s)=x_0\not=0$ is a constant function, being $z(s)=\pm s+z_0$ for some $z_0\in\r$. Then  $\gamma$ is a vertical line. In this case, 
$$\Psi(s,t)=(x_0 \cos t,x_0\sin t,ht\pm s+z_0).$$
This surface is a circular cylinder of radius $|x_0|$.  Now we have
\begin{equation*}
\begin{split}
A_0&=x_0^2z v_3, \\
A_1&=x_0^2h v_3,\\
A_2&=x_0^3v_2(1+\alpha),\\
A_3&= x_0^3v_1(1+\alpha)\\
\end{split}
\end{equation*}
Then $A_1=0$ implies $v_3=0$. This means that the $z$-axis, which it is the twist axis of the helicoidal surface, lies contained in the vector plane $\{p\in\r^3\colon\langle p,\vec{v}\rangle=0\}$. Since $v_1$ and $v_2$ cannot be simultaneously $0$, from $A_2=0$ or $A_3=0$ we deduce  $\alpha=-1$. 
\item Case $v_3=0$. The linear combination $v_2 A_3-v_1A_2=0$ becomes
$$h\alpha\cos\theta(x^2+h^2\cos^2\theta)=0.$$
Since $h\not=0$, this equation  implies $\cos\theta=0$ identically. This case has been   considered in the previous item. This completes the proof of Thm. \ref{t1}.
\end{enumerate}

\section*{Acknowledgements}
The author has been partially supported by Grant PID2023-150727NB-I00 funded by MICIU/AEI/10.13039/501100011033, and ERDF/EU and 
Grant PID2023-150727NB-I00 and Maria de Maeztu Unit of Excellence IMAG, reference CEX2020-
001105-M, funded by MICIU/AEI/10.13039/501100011033, and ERDF/EU. 

 \end{document}